\documentclass{amsart}

\usepackage{amssymb,amsmath,amsfonts,amsthm}
\usepackage[all]{xy}
\usepackage{enumerate}
\usepackage{mathrsfs}
\usepackage{graphicx}
\usepackage{url}

\newtheorem{thm}{Theorem}
\newtheorem{cor}[thm]{Corollary}
\newtheorem{lem}[thm]{Lemma}

\newtheorem{prop}[thm]{Proposition}

\newtheorem{rem}[thm]{Remark}

\def\dto{\dashrightarrow}

\newcommand{\llar}{-\kern-5pt-\kern-5pt\longrightarrow}
\newcommand{\surjects}{\twoheadrightarrow}

\def\leq{\leqslant}
\def\geq{\geqslant}
\def\Bc{{\mathcal{B}}}
\def\RR{{\mathbb{R}}}

\newcommand\ZZ{{\mathbb Z}}

\newcommand\PP{{\mathbb P}}

\newcommand\Rc{{\mathcal R}}

\newcommand\Jc{\mathcal J}

\newcommand\Fc{\mathcal F}

\newcommand{\pp}{\mathfrak{p}}

\DeclareMathOperator\sat{sat}
\DeclareMathOperator\Syz{Syz}

\def\pp{{\mathfrak{p}}}
\def\qq{{\mathfrak{q}}}

\def\coker{{\mathrm{coker}}}

\def\rank{{\mathrm{rank}}}
\def\Homgr{{\mathrm{Homgr}}}
\def\ind{{\mathrm{in}}}

 \def\ff{{\bf f}}
 \def\gg{{\bf g}}
 \def\hh{{\bf h}}

 \def\xx{{\bf x}}
 \def\yy{{\bf y}}
 \def\zz{{\bf z}}

 \def\PP{{\bf P}}

\newcommand{\fn}{\mathfrak{n}}

\newcommand{\fa}{\mathfrak{a}}
\newcommand{\fb}{\mathfrak{b}}
\newcommand{\fc}{\mathfrak{c}}

\title{Effective criteria for bigraded birational maps}

\author[N. Botbol]{Nicol\'{a}s Botbol}
\address{Departamento de Matem\'atica, FCEN, Universidad de Buenos Aires, Argentina }
\email{nbotbol@dm.uba.ar}
\urladdr{http://mate.dm.uba.ar/~nbotbol/}

\author[L. Bus\'e]{Laurent Bus\'e}
\address{INRIA Sophia Antipolis
EPI Galaad, 
2004 route des Lucioles,
06902 Sophia Antipolis, France.}
\email{Laurent.Buse@inria.fr}
\urladdr{http://www-sop.inria.fr/members/Laurent.Buse/}

\author[M. Chardin]{Marc Chardin}
\address{Institut de Math\'ematiques de Jussieu.
UPMC,  4 place Jussieu,
75005 Paris, France}
\email{marc.chardin@imj-prg.fr}
\urladdr{http://webusers.imj-prg.fr/~marc.chardin/}

\author[S. H. Hassanzadeh]{Seyed Hamid Hassanzadeh}
\address{Instituto de Matematica, Universidade Federal do Rio de Janeiro, Brazil.}
\email{hamid@im.ufrj.br}

\author[A. Simis]{Aron Simis}
\address{Departamento de Matem‡tica
Universidade Federal de Pernambuco
CEP 50740-560, Recife, Pernambuco, Brazil.}
\email{aron@dmat.ufpe.br}
\urladdr{www.abc.org.br/~aron}

\author[Q. H. Tran]{Quang Hoa Tran}
\address{Hue University's College of Education, 34 Le Loi St., Hue City, Vietnam
\& Institut de Math\'ematiques de Jussieu.
UPMC, 4 place Jussieu,
75005 Paris, France}
\email{quang-hoa.tran@imj-prg.fr}
\urladdr{http://webusers.imj-prg.fr/~quang-hoa.tran/}

\date{\today}

\begin{document}

\maketitle

\begin{abstract}
	In this paper, we consider rational maps whose source is a product of two subvarieties, each one being embedded in a projective space. Our main objective is to investigate birationality criteria for such maps.  First, a general criterion is given in terms of the rank of a couple of matrices that became to be known as \emph{Jacobian dual matrices}. Then, we focus on rational maps from $\PP^1\times \PP^1$ to $\PP^2$ in very low bidegrees and provide new matrix-based birationality criteria by analyzing the syzygies of the defining equations of the map, in particular by looking at the dimension of certain bigraded parts of the syzygy module. Finally, applications of our results to the context of geometric modeling are discussed at the end of the paper. 
\end{abstract}


\section{Introduction}

 A rational map $\Fc:\PP^r \dasharrow \PP^s$ between projective spaces is defined by an ordered set of homogeneous polynomials $\ff:=\{f_0,\ldots,f_s\}$  in $r+1$ variables, of the same degree and not all zero. The problem of providing sufficient conditions for such a map $\Fc$ to be birational has attracted much interest in the past and it is still an active area of research. For computational purposes, methods  based  on the nature of the syzygies of $\ff$ are the most suitable in the sense of effective results in the usual implementation of the Gr\"{o}bner basis algorithm. This syzygy-based approach goes back to \cite{SHK} where sufficient conditions for birationality were given in the case $r=s$. Then, several improvements have been introduced in relation with the equations of the symmetric and the Rees algebras of the ideal generated by $\ff$ \cite{RS,S}, including in arbitrary characteristic \cite{DHS}, and also in relation with  the fibers of $\Fc$ \cite{EU}.    

\medskip
 
In this paper, we aim to extend some of these methods and techniques to the context of rational maps whose source is a product of two projective spaces $\PP^n\times \PP^m$ instead. These maps are defined by an ordered set of bihomogeneous polynomials in two sets of $n+1$ and $m+1$ variables, respectively. 
 For the sake of emphasis, we call them {\em bigraded rational maps}. 
 Modern important motivation for considering bigraded rational maps comes from the field of geometric modeling. Indeed, the geometric modeling community uses almost exclusively bigraded rational maps for parameterizing surfaces,  dubbing such maps  {\em rational tensor-product B\'ezier parameterizations}. It turns out that an important property is to guarantee the birationality of these parameterizations onto their images. An even more important property is to preserve this birationality property during a design process, that is to say when the coefficients of the defining polynomials are continuously modified. As a first attempt to tackle these difficult problems, we will analyze in detail birational maps from $\PP^1\times \PP^1$ to $\PP^2$ in low bidegrees by means of syzygies.

\medskip

Through its various sections, this paper traverses topics from algebra to geometry and to modeling. In Section \ref{sec:general-criterion}, a general criterion for characterizing bigraded birational maps is proved by means of algebraic tools. It is based on the rank of two matrices, called \emph{Jacobian dual matrices}, that are built from some particular equations of the Rees algebra of the bihomogeneous equations defining the rational map. This criterion is actually an analogue of the existing Jacobian dual criterion of rational maps between varieties embedded in projective spaces \cite{RS,S,DHS}. 
  
In Section \ref{sec:lowdegdim}, we turn to a more geometric language since the bigraded rational maps are investigated through the properties of their base locus. By focussing on bigraded birational maps from $\PP^1\times \PP^1$ to $\PP^2$, we obtain very simple birationality criteria in bidegree $(1,1)$ and bidegree $(1,2)$ in terms of the dimension of some bigraded parts of the syzygies of the equations defining the rational map. Another important contribution of our work is a detailed study of the case of bidegree (2,2) maps for which we provide a complete listing of possible birational maps.

Finally, in Section \ref{sec:applications} we investigate applications of our results to the field of modeling. In particular, for bigraded plane rational maps of bidegree (1,1) and (2,1) we explain how some particular coefficients of the map, called the \emph{weights} of the parameterization, can be tuned in order to obtain a birational map. It is important to notice that the inverse map is then given by explicit minors from the matrix characterizing the birationality of the map. In the bidegree (2,1) case, our new birationality criterion allows to assign the control of this tuning to a structured low-rank matrices approximation algorithm, in the context of numerical computations.

\section{General birationality criterion}\label{sec:general-criterion}

In this section, we provide a general effective criterion for birationality of a bigraded rational map with source a biprojective space $ \PP^n\times  _k\PP^m$. We will state the results under  more general hypotheses, namely, when the source is a product $X\times_k Y$, where $X\subset \PP^n$, $Y\subset \PP^m$ denote non-degenerate irreducible  projective varieties over an algebraically closed field $k$.  The criterion is an analogue of the so-called Jacobian dual criterion which has been studied so far in the context of a rational maps between varieties embedded in projective spaces \cite{RS, S,DHS}.

\subsection{Birationality and bigraded Rees algebras}

As in the case of a rational map between projectively embedded varieties, where the notion of the graph of the map is encoded in taking the Rees algebra of an equigenerated homogeneous base ideal, a rational map with source a multi-projectively embedded variety and target a projectively embedded variety has a graph encoded in taking the Rees algebra of an equigenerated multihomogeneous ideal.
As for a rational map with source a projectively embedded variety and target a  multi-projectively embedded variety, the algebraic object that conveniently encodes the graph is a multi-Rees algebra -- i.e., the Rees algebra of a module which is the direct sum of a finite set of equigenerated homogeneous ideals of various degrees.

Although valid in the arbitrary multigraded case, for simplicity, we state it in the biprojective case.
Thus, let $X\subset \PP^n$, $Y\subset \PP^m$ and $Z\subset \PP^s$ denote non-degenerate irreducible  projective varieties over an algebraically closed field $k$. Let $A=k[\xx]=k[x_0,\ldots,x_n]/\fa$, $B=k[\yy]=k[y_0, \ldots,y_m]/\fb$ and $S=k[\zz]=k[z_0,\ldots, z_s]/\fc$ stand for the respective homogeneous coordinate rings. We also denote $R:=A\otimes_k B\simeq k[\xx,\yy]/(\fa,\fb)$. A rational map $\mathcal{F}:X\times Y \dto Z$ is  defined by bihomogeneous
polynomials $f_0(\xx,\yy),\ldots,f_s(\xx,\yy)$ in $R$ of fixed bidegree $(a,b)$, not all zero.
We say that $\mathcal{F}$ is birational with image $Z$ if it is dominant and admits an inverse rational map with image $X\times_kY$. Note that the inverse map is necessarily given by a pair of rational maps $Z\dto X$ and  $Z\dto Y$ defined by homogeneous polynomials $\gg:=\{g_0,\ldots,g_n\}$ and $\hh:=\{h_0,\ldots,h_m\}$ of fixed degrees $d_1$ and $d_2$, respectively.

\begin{lem}\label{Lgraph}
With the above notation, set $I:=(f_0(\xx,\yy),\ldots,f_s(\xx,\yy))\subset R$ and $J_1:=(\gg)$, $J_2:=(\hh)$.
Then the identity map on $k[\xx,\yy,\zz]$ induces a $k$-algebra isomorphism between the Rees algebra $\Rc_R(I)$ and the multi-Rees algebra $\Rc_S(J_1\oplus J_2)$.
\end{lem}
\begin{proof} The proof is tailored on the one in \cite[Theorem 2.1]{S} (see also \cite[Theorem 2.18]{DHS}).
Consider a polynomial presentation
$$R[\zz]=\frac{k[\xx,\yy]}{(\fa,\fb)}[\zz]\surjects R[{\ff}t]=\frac{k[\xx,\yy]}{(\fa,\fb)}\,[{\ff}t]=\Rc_R(I),\;\;
 z_k\mapsto {f_k}t$$
whose restriction to $R=k[\xx,\yy]/(\fa,\fb)$ is the identity.
Let $\overline{\mathcal{J}}=(\Jc,\fa,\fb)/(\fa,\fb)$ denote the kernel, with $\Jc$ the ideal generated by the $\zz$-homogeneous polynomials, with bihomogeneous polynomial coefficients in $k[\xx,\yy]$, vanishing on $f_0(\xx,\yy),\ldots,f_s(\xx,\yy)$ modulo $(\fa,\fb)$.

Note that $\fc\subset \Jc$.
Indeed, taking $\ff$ as homogeneous polynomials for the total degree of their fixed bidegree, it is clear that the image $Z$ is identified with ${\rm Proj}(k[\ff])$ up to degree normalization.
Since the two algebras $k[\ff]$ and $k[\ff t]$ are $k$-isomorphic as graded algebras and $\ker (k[\zz]\surjects k[\ff t])\subset \Jc$, we are through.
In particular, the Rees algebra  $\Rc_R(I)$ is a residue $k$-algebra of $R[\zz]/\bar{\fc}=k[\xx,\yy,\zz]/(\fa,\fb,\fc)$.

By the same token, one has
$$\Rc_S(J_1\oplus J_2)\simeq S[\gg u,\hh v]\simeq k[\zz][\xx,\yy]/(\fc,\Jc'),$$
 where  $\Jc'$ is generated by those $\xx,\yy$-bihomogeneous polynomials with homogenous  coefficients in $k[\zz]$ vanishing on both sets $\gg$ and $\hh$ modulo $\fc$.
Similarly, both $\fa$ and $\fb$ are contained in $\Jc'$ -- for example, note for this, that a form of degree $d$ in $k[\xx]$ is a bihomogeneous polynomial in $k[\xx,\yy]$ of bidegree $(d,0)$ with homogeneous coefficients in $k[\zz]$ of degree $0$.

Thus, $\Rc_S(J_1\oplus J_2)$ is  a residue $k$-algebra of $S[\xx,\yy]/(\fa,\fb)=k[\xx,\yy,\zz]/(\fa,\fb,\fc)$ as well.

We now claim that the identity map of $k[\xx,\yy,\zz]/(\fa,\fb,\fc)$ induces the required $k$-algebra isomorphism, for which it suffices now to show that $\Jc'\subset (\fa,\fb,\Jc)$ and that $\Jc\subset (\fc,\Jc')$.

Let $F(\zz,\xx,\yy)=\sum_i p_i(\zz)\xx^{\alpha_i}\yy^{\beta_i} \in \Jc'$, where $|\alpha_i|=p$ and $|\beta_i|=q$ for all $i$. By the definition of $\Jc'$, one has $F(\zz,\gg,\hh)\in \fc$. Therefore,  thus $F(\ff,\gg,\hh)=0\in R$. On the other hand since the pair $(\gg;\hh)$
defines the inverse map to $\mathcal F$, there exist forms $D$ and $D'$ in $R\setminus{0}$, perhaps of different degrees, such that $\gg(\ff)=\xx\,D$ and $\hh(\ff)=\yy\,D'$. It follows that   $F(\ff,\gg,\hh)=(\sum_i p_i(\ff)\xx^{\alpha_i}\yy^{\beta_i})D^a D'^b$.
Since $R$ is an integral domain, the  vanishing of the latter shows that $F(\ff,\xx,\yy)=0$ on $R$. In particular $F \in (\fa,\fb,\Jc)$.

The other inclusion is obtained by a similar argument.
\end{proof}

\subsection{Bi-graded Jacobian dual criterion}
We are now ready to present a multiprojective version of the Jacobian dual criterion of birationality. For simplicity, we stick to the biprojective case, as the arbitrary multiprojective case requires only a small set of changes.

We will focus on the presentation ideal $(\fa,\fb,\Jc)\subset k[\xx,\yy,\zz]$ of the Rees algebra $\Rc_R(I)$.
Consider the elements of degrees $(1,0,\ast)$  and $(0,1,\ast)$ in  $(\fa,\fb,\Jc)$, where $\ast$ denotes an arbitrary degree in $\zz$. Since by assumption $X$ and $Y$ are non-degenerate
these elements belong to the graded pieces $\Jc_{(1,0,\ast)}$ and $\Jc_{(0,1,\ast)}$, respectively.
Now, a form of degree $(1,0,\ast)$ can be thought as a form of bidegree $(1,*)$ in $k[\xx,\zz]$.
Moreover, since $X$ is nondegenerate, each such form has a unique expansion of the shape $\sum_iQ_i(\zz)x_i$, where $Q_i(\zz)\in k[\zz]$ is homogeneous of degree $*$.
Considering these expansions for a minimal set of generating forms of the ideal $(\Jc_{(1,0,\ast)})$ and taking the corresponding matrix of $\xx$-derivatives yields a weak Jacobian dual matrix $\Psi_{\xx}$ in the sense of \cite[Section 2.3]{DHS} -- here dubbed an {\em $\xx$-partial Jacobian dual matrix}. We similarly introduce an {\em $\yy$-partial Jacobian dual matrix}  $\Psi_{\yy}$. Finally, thinking of these matrices as maps over $k[\zz]$, we denote by $\Psi_{\xx}\otimes _{k[\zz]}S$ and $\Psi_{\yy}\otimes _{k[\zz]}S$ the respective maps obtained modulo $\fc$.

\begin{thm}\label{Pjacobiandual} With the previous notation, the rational map  $\mathcal{F}:X\times Y \dto Z$   is birational with image $Z$ if and only if $\rank_S(\Psi_{\xx}\otimes _{k[\zz]}S)=n$
 and $\rank_S(\Psi_{\yy}\otimes _{k[\zz]}S)=m$.
In addition, both halves of the expression of the inverse of
$\Fc$ are given by {\rm (}signed{\rm )}  ordered maximal minors of an $n\times (n+1)$ submatrix of $\Psi_{\xx}$ and of an $m\times (m+1)$ submatrix of $\Psi_{\yy}$, respectively.

\end{thm}
\begin{proof}Suppose that $\mathcal{F}$ is birational with image $Z$.
By the proof of Lemma~\ref{Lgraph},  in particular $\Jc_{(1,0,\ast)}=\Jc'_{(\ast,1,0)}$ -- notice that $\ast$, $1$ and $0$ are the respective degrees in $\zz$,$\xx$ and $\yy$.
This implies that $\Psi_{\xx}$ can also be written in terms of $\Jc'_{(\ast,1,0)}$.
But, as such and due to the definition of $\Jc'$, we get an equality
$$
I_1([x_0\cdots x_n]\cdot(^t\Psi_{\xx}\otimes S))=I_1([x_0\cdots x_n]\cdot {\rm Syz}_S(\gg)),
$$ 
where $ {\rm Syz}_S(\gg)$ denotes the matrix of syzygies of  $(\gg)\subset S$. Since neither  $^t\Psi_{\xx}\otimes S$ nor ${\rm Syz}_S(\gg)$  involves any variables other than $\zz$, it then follows that these matrices define the same column space and hence have the same rank. But, clearly  $\rank_S({\rm Syz}_S(\gg))=n$. A similar argument applies to the $\yy$ part.

The proof of the converse statement is the same as the proof for the projective varieties in \cite{DHS}, with the obvious adaptation. Thus, let $M$ denote an $n\times (n+1)$ submatrix of $\Psi_{\xx}$ which is of rank $n$ over $S$. Let $\Delta_0(\zz),\cdots,\Delta_n(\zz)$ be its ordered signed minors. By the Hilbert--Koszul lemma \cite[Proposition 2.1]{DHS}, the vector $\Delta_i(\zz)e_j-\Delta_j(\zz)e_i$ belongs to the column space of $M$ and hence to that of $\Psi_{\xx}$. The fact that $I_1([x_0\cdots x_n]\cdot\Psi_{\xx})=(\Jc_{(1,0,\ast)})$ ensures that $x_i\Delta_j(\zz)-x_j\Delta_i(\zz)\in (\Jc_{(1,0,\ast)})$. In particular
$x_i\Delta_j(\ff)-x_j\Delta_i(\ff)=0$ in $R$. We claim that the $n+1$-tuple $(\Delta_0(\ff)\,\cdots\,\Delta_n(\ff))$ does not vanish on $R$.  To see this, recall that the homogeneous coordinate ring of the image of $\Fc$ is $k[f_0,\cdots,f_s]\simeq S$ (up to degree normalization). Since $S=k[\zz]/\fc$, then $\Delta_i(\ff)=0\in R$ if and only if $\Delta_i(\zz) \in \fc$, i.e., if and only if $\Delta_i(\zz)=0\in S$. But this cannot happen for all $i$ because $\rank(M\otimes_{k[\zz]}S)=n$.
It now follows that the rational map $Z\dto X$
defined  by $(\Delta_0(\zz):\cdots:\Delta_n(\zz))$ gives the first half of the inverse to $\Fc$.
The second half is treated entirely in the same way.
\end{proof}


\subsection{Linear syzygies and birationality}

Theorem \ref{Pjacobiandual} yields an explicit criterion for deciding if a given bigraded rational is birational. This criterion relies on Gr\"obner basis computations in order to get the equations of a Rees algebra. Therefore, it may suffer from limitations with complicated examples and especially, it does not allow to treat a family of rational maps at once (for Gr\"obner basis computations are not stable under change of basis). In order to work around these two drawbacks, we investigate how birationality can be detected by means of syzygies of the ideal $I$ generated by the coordinates of the rational maps, instead of the whole collection of equations of $I$. Indeed, any criterion based on some syzygies of given bidegree of $I$ will only rely on linear algebra computations.   

We will need to consider not just Rees algebras of ideals or multi-Rees algebras, but the full notion of the Rees algebra of a module as discussed in \cite{ram1}.

The following proposition gives an analogue of \cite[Theorem 3.2]{DHS}.

\begin{prop}\label{Plinear rank} Let $\mathcal{F}:\PP^n\times \PP^m \dto \PP^s$ stand for  a rational map defined by bihomogeneous
polynomials $f_0(\xx,\yy),\cdots,f_s(\xx,\yy)$ in $R:=k[\xx,\yy]$ and set $I:=(f_0(\xx,\yy),\cdots,f_s(\xx,\yy))$.
If the image of $\Fc$ has dimension $n+m$ and the submatrix of the syzygy matrix of $I$ consisting of columns of bidegrees $(1,0)$ and $(0,1)$ has rank $s$ {\rm (}maximal possible{\rm )}, then $\Fc$ is birational onto its image.
\end{prop}
\begin{proof}
Note that the image of $\Fc$ is a projective subvariety of $\PP^s$.
Let  $\zz$ be homogeneous coordinates on $\PP^s$ and set $S=k[\zz]/\fc$ for the homogeneous coordinate ring of the image of $\mathcal{F}$. 
Since $I$ is bihomogeneous, it admits a minimal syzygy matrix whose columns are bihomogeneous.
Clearly, the independent syzygies of degrees either $(1,0)$ or $(0,1)$ will be columns of this matrix.
Let $M$ denote the submatrix with these columns.
Then choose a matrix $N$  with entries in $k[\zz]$ such that 
$[\zz]\cdot M=[\xx,\yy]\cdot N$. Let $ E=\coker (N\otimes_{k[\zz]}S)$.

We now introduce in the discussion the Rees algebras  $\Rc_R(\coker \,M)$ and $\Rc_S(E)$. 
Thus, one has
 $$\Rc_R(\coker\, M)=\frac{k[\xx,\yy,\zz]}{( [\zz]M,\tau_1)}~~~
\text{and}~~~~\Rc_S(E)=\frac{k[\xx,\yy,\zz]}{(\fc, [\xx,\yy]N,\tau_2)}$$
where $\tau_1$ is the $R$-torsion of ${\rm Sym}_R(\coker M)$ lifted to $k[\xx,\yy,\zz]$ and, similarly,  $\tau_2$ is the $S$-torsion of  ${\rm Sym}_S(E)$ lifted to $k[\xx,\yy,\zz]$. 

Note that, by definition, the Rees algebra $\Rc_S(E)$ of  $E$ and that of $E$ modulo its torsion coincide. Since $S$ is a domain, the latter module embeds into a free module over $S$. In particular, $\Rc_S(E)$ is a domain, i.e., $(\fc, [\xx,\yy]N,\tau_2)$ is a prime ideal. 

 We claim that $( [\zz]M,\tau_1)\subseteq (\fc, [\xx,\yy]N,\tau_2).$

Indeed, let $G=G(\xx,\yy,\zz)\in \tau_1$.
Then there exists $F(\xx,\yy)\in k[\xx,\yy]\setminus \{0\}$ 
such that $F(\xx,\yy)G\subset ( I_1(\zz \cdot M))\subset ({\mathfrak c}, [\xx,\yy]N,\tau_2)$.
If $G\not\in ({\mathfrak c}, [\xx,\yy]N,\tau_2)$ then $F(\xx,\yy)\in ({\mathfrak c}, [\xx,\yy]N,\tau_2)$.
By the definition of $\tau_2$ there exists $H(\zz)\in k[\zz]\setminus {\mathfrak c}$
such that $H(\zz)F(\xx,\yy)\in ({\mathfrak c}, [\xx,\yy]N)$. Recall that $[\zz]\cdot M=[\xx,\yy]\cdot N$.
Evaluating $\zz \mapsto \ff$  would give $H(\ff)F(\xx,\yy)= 0$ whence $F(\xx,\yy)=0$
since $H(\ff)\neq 0$; this is a contradiction.

As a consequence, one has a surjective $R$-algebra map
 $\Rc_R(\coker\, M)\twoheadrightarrow \Rc_S(E)$ and hence 
 \begin{equation}
 \dim(\Rc_S(E))\leq \dim(\Rc_R(\coker\, M)). 	
 \end{equation}
 Now
 $\dim(\Rc_S(E))=\dim(S)+(n+1+m+1)-\rank(N\otimes S)$ and  $\dim(\Rc_R(\coker\, M))=\dim(R)+s+1-\rank(M)$. Since $\dim(R)=\dim(\PP^n\times \PP^m)+2$ we have $\dim(R)-\dim(S)=1$. Therefore the above inequality
 implies that
$$n+m+(\rank(M)-s)\leq \rank(N\otimes S).$$

Since $\rank(M)=s$ by assumption, we obtain $n+m\leq  \rank(N\otimes S)$. Notice that $N\otimes S$ is a submatrix of the ``concatenated'' Jacobian dual matrix 
 $$\rho:=\left(\begin{array}{c}
 	\Psi_{\xx}\otimes S  \\ \hline 
	 \Psi_{\yy}\otimes S
 \end{array}\right)$$
 in the notation introduced in the previous subsection.
 
 Thus we have
$\rank(\rho)\geq n+m$, whenever  $\rank(M)=s$. 

{\sc Claim:}  $\rank(\Psi_{\xx}\otimes S)\leq n$ (and, similarly,
   $\rank(\Psi_{\yy}\otimes S)\leq m$).  
   
Assuming the claim, it follows that $\rank(\rho)\leq n+m$ and the equality happens if and only if  $\rank(\Psi_{\xx}\otimes S)= n$
 and  $\rank(\Psi_{\yy}\otimes S)= m$.
Therefore, the result follows from Proposition~\ref{Pjacobiandual}.

We now  show that  $\rank(\Psi_{\xx}\otimes S)\leq n$. Indeed, consider the field $K:=k(\yy)$ (the generic point of $\PP^m_k$) and the rational map 
$$\Fc':\PP^n_K\dto \PP^s_K$$
which is defined by the polynomials $f_0,\ldots,f_s$ viewed as polynomials in $K[\xx]$. Let $S':=K[\yy]/(\fc)$ be the coordinate ring of the image of $\Fc'$ and consider the Jacobian dual matrix of $\Fc'$ over $S'$ : $\Psi'\otimes S'$. Then, because of the field inclusion $k\hookrightarrow K$ the column space of $\Psi_{\xx}\otimes S$ is contained in the column space of $\Psi'\otimes S$. Therefore, $\rank(\Psi_{\xx}\otimes S)\leq \rank(\Psi'\otimes S')\leq n$ where the last inequality follows from \cite[Corollary 2.16]{DHS}.
\end{proof}

\begin{rem}\rm
The mutual independence of the hypotheses in Proposition \ref{Plinear rank} has already been observed in \cite[bottom p. 409]{DHS} in the case the source of $\Fc$ is a single projective space; likewise, in our setting. The most obvious situation where the number of linear syzygies of the required type is maximal and yet the image has smaller dimension is obtained as follows. We explain the projective version, the biprojective one being entirely similar.

Let $\mathcal{F}:\PP^r\dasharrow \PP^s$ be a birational map onto the image such that the linear syzygies of the defining forms $\ff$ have maximal rank.
Let $I\subset R$ denote the base ideal of $\mathcal{F}$.
Consider the coordinate projection $\pi:\PP^{r+1}\dasharrow \PP^r$ defined by the first $r+1$ variables -- thus, this corresponds to the ring extension $R=k[x_0,\ldots, x_r]\subset R[x_{r+1}]$.
Since the latter is a faithfully flat extension, or directly, the module of syzygies of $\ff$ on $R[x_{r+1}]$ is extended from the $R$-module of syzygies of $\ff$, in particular the linear parts have the same rank as $R$-module or $R[x_{r+1}]$-module.
At the other end the $k$-algebra $k[\ff]$ is the same whether considered as a subalgebra of $R$ or of $R[x_{r+1}]$.
Therefore, the composite map $\mathcal{F}\circ \pi:\PP^{r+1}\dasharrow \PP^s$ has the same linear rank and the same image as $\mathcal{F}$.
This shows that maximal linear rank does not imply maximal dimension of the image. 
\end{rem}
To get a biprojective analogue, it suffices to take a one-sided projection $\PP^{n+1}\times \PP^m\dasharrow \PP^{n}\times \PP^m$ to the source of a birational map $\PP^{n}\times \PP^m\dasharrow \PP^s$ having maximal linear rank in the sense of the statement of the Proposition (e.g., an arbitrary Segre map).

\smallskip 

It is of course clear that the full converse of the statement in the proposition is false. In the projective case, one can take a birational parameterization $\PP^1\dasharrow \PP^2$ of a plane curve with parameters of degree $\geq 4$ (hence, with linear rank $0$).
For example, take the parameters $x^4, y^4, x^3y+xy^3$ on $k[x,y]$.
Since the image is a quartic curve, the map $\mathcal{F}$ defined by these parameters is automatically birational onto the curve.

To extract a biprojective example, compose the induced map $$({\rm id},\mathcal{F}):\PP^1\times \PP^1\dasharrow \PP^1\times \PP^2$$ with the Segre map $\PP^1\times \PP^2\dasharrow \PP^5$. The result is clearly birational onto a subvariety of dimension $2$ of the Segre embedding.
However, a calculation with \texttt{M2} shows that the linear rank is only $3$.
 
\medskip


%
%
%
%

\section{Syzygies of low degree of bigraded maps in the plane}\label{sec:lowdegdim}

In this section, we will focus on the linear syzygies of bigraded rational maps from $\PP^1\times \PP^1$ to $\PP^2$. 
Under consideration will be the cases where the total degree of the biforms is $2$ or $3$.
Note that in the projective case, plane Cremona maps of these degrees are automatically de Jonqui\`eres maps. In both cases the base ideal is an ideal of $2$-minors of a $3\times 2$ matrix, with two linear syzygies or a linear syzyzy and a quadratic one, respectively \cite{HS}.

In the case of a bigraded rational map defined by polynomials $\ff:=\{f_0,f_1,f_2\}$ of bidegree (1,1) it is very easy to see that, up to linear transformations in the source and target spaces, there are only two maps :
$$\PP^1\times \PP^1 \rightarrow \PP^2 : (x:y)\times(u:v) \mapsto (xu:yu:yv),$$
$$\PP^1\times \PP^1 \rightarrow \PP^2 : (x:y)\times(u:v) \mapsto (xu:yu:xv+yu).$$
The first one is birational and $\ff$ has two minimal syzygies of respective bidegrees $(1,0)$ and $(0,1)$, whereas the second one is not birational and $\ff$ has exactly five linearly independant minimal syzygies. Therefore, birationality is here guaranteed by the existence of a linear syzygy. To understand to which extent such a result can be generalized to higher bidegree, some preliminary work is required.
Our tools will be largely homologically oriented. Before going into the details, we first fix some notation.

\smallskip

We will switch from the previous notation $\mathcal{F}$ for a rational map to the symbol $\phi$.
Let $k$ be an infinite field. Let $R:=k[x,y;u,v]$ be the bigraded polynomial ring with weights defined by $\deg(x)=\deg(y)=(1,0)$ and $\deg(u)=\deg(v)=(0,1)$. Let $f_0,f_1,f_2$ be three bihomogeneous polynomials of bidegree $(a,b)$ and set $I=(f_0,f_1,f_2)\subset R$. Consider the rational map defined by these forms:
\begin{eqnarray*}\label{eq:phi}
	\phi: \PP^1\times \PP^1 & \dasharrow & \PP^2 \\
	((\alpha:\beta), (\gamma:\delta)) & \mapsto & (f_0(\alpha,\beta, \gamma,\delta):f_1(\alpha,\beta, \gamma,\delta):f_2(\alpha,\beta, \gamma,\delta)).
\end{eqnarray*}
We assume throughout that $\phi$ is a dominant rational map and that the polynomials $f_0,f_1,f_2$ do not have a proper common factor in $R$, which, in a more geometric terminology, means that these polynomials define a zero-dimensional scheme in $\PP^1\times \PP^1$; let $\Bc$ denote this scheme -- called the \emph{base scheme} of $\phi$.  
We note  that the degree of $\Bc$, denoted $\deg(\Bc)$ is equal to the bigraded Hilbert function of $R/I$ for sufficiently high bidegree $(\mu,\nu)$.  

In analogy to a well-known degree formula in the projective case, one has the following degree formula in the biprojective counterpart (see, e.g., \cite[Lemma 7.4]{Bot11}):
\begin{equation}\label{degree-formula}
\deg (\phi)=2ab-\sum_{x\in \Bc} e_x(I),
\end{equation}
where $\deg(\phi)$ stands for the field degree of the rational map $\phi$ and $e_x(I)$ stands for the Hilbert-Samuel multiplicity of $I$ on the localization $R_{\pp}$ at the defining prime ideal $\pp$ of the point $x$ (see \cite[\S4.5]{BH93} for more details). An important property of this latter multiplicity is that it is equal to the length of the residue of $R_{\pp}$ modulo the ideal generated by two general $k$-linear combinations of the polynomials $f_0,f_1,f_2$.

In particular, the degree formula \eqref{degree-formula} can be easily derived from this property as follows. Two linear forms in three variables define a point $P$ in the target $\PP^2$ and the corresponding linear combinations of  $f_0,f_1,f_2$ define a subscheme $Y_P$ in $\PP^1\times \PP^1$ giving the inverse image of $P$ by $\phi$ off $\Bc$. On an open subset of $\PP^2$, or equivalently of the space of coefficients of the linear forms, the inverse image is a finite set. Hence, for a general point $P$, $Y_P$ is a complete intersection of degree $2ab$ and is the union of a component $X_P$ not meeting $\Bc$, of degree equal to the degree of the map (notice that $X_P$ is reduced by Bertini theorem) and a component with support $\Bc$ in which each point $x$ has multiplicity equal to $e_x(I)$. Indeed, the multiplicity at a point $x$ is constant and equal to its minimal value for two linear forms corresponding to a dense open subset of $\PP^2$; this value is $e_x(I)$ by \cite[Corollary 4.5.10]{BH93}.

\subsection{Counting linear syzygies} 

We denote by $\Syz(I)\subseteq R^3$ the module of syzygies of $I$. It is a bigraded module and the \emph{linear syzygies} correspond to the graded parts $\Syz(I)_{(1,0)}$ and $\Syz(I)_{(0,1)}$. In other words, in the structural bigraded exact sequence
$$0\longrightarrow Z_1\longrightarrow R^3(-a,-b) \xrightarrow{(f_0,f_1,f_2)} I\longrightarrow 0,$$
we have the identification $\Syz(I)=Z_1(a,b).$ In the sequel, we will use the notation $K_\bullet$, $Z_\bullet$, $B_\bullet$ and $H_\bullet$ to refer to the terms, cycles, boundaries and homology modules of the Koszul complex of the sequence $f_0,f_1,f_2$. We set $\fn:=(x,y)\cap (u,v)=(xu,xv,yu,yv) \subset R$ for the ideal generated by all monomials of bidegree $(1,1)$.
Recall the following bigraded exact sequence in local cohomology
{\small
\begin{equation}\label{eq:les}
	0\longrightarrow H_\fn^0(R/I) \longrightarrow R/I \longrightarrow \bigoplus_{(\mu,\nu)\in \ZZ^2} H^0(\PP^1\times \PP^1, \mathcal{O}_{_{\Bc}}(\mu,\nu))\longrightarrow H_\fn^1(R/I)\longrightarrow 0.
\end{equation}
}

In the following, an upper right star $*$ attached to an $R$-module will denote its Matlis dual.   

\begin{lem} \label{lemma2} Set $\Omega:= \{  (\mu,\nu)\in \ZZ^2  \,| \, -2b< b\mu-a\nu< 2a \}.$ Then
\begin{equation}\label{bigraded-piece)}
H_\fn^1(R/I)_{(\mu,\nu)} \simeq  (H_1)^\ast_{(3a-\mu-2,3b-\nu-2)},
\end{equation} 
for every $(\mu,\nu)\in \Omega$
\end{lem}
\begin{proof} The argument hinges on the two spectral sequences associated to the double complex $C_q^p=C_\fn^p(K_q)$, where $K_\bullet$ the Koszul complex of the sequence $f_0,f_1,f_2.$ 
	One of them abuts at step two with:
	\begin{displaymath}
		\begin{array}{cccc}
			H_\fn^0(H_3) & H_\fn^0(H_2) & H_\fn^0(H_1) & H_\fn^0(H_0) \\
			H_\fn^1(H_3) & H_\fn^1(H_2) & H_\fn^1(H_1) & H_\fn^1(H_0)\\
			0 & 0& 0& 0\\
			0 & 0& 0& 0
		\end{array}
	\end{displaymath}
	The other one gives at step one:
	\begin{displaymath}
		\begin{array}{ccccccc}
			0 &\longrightarrow & 0&\longrightarrow & 0& \longrightarrow & 0\\
			0 &\longrightarrow & 0&\longrightarrow & 0& \longrightarrow & 0\\
			H_\fn^2(K_3) & \longrightarrow & H_\fn^2(K_2) & \longrightarrow & H_\fn^2(K_1) &\longrightarrow  & H_\fn^2(K_0) \\
			H_\fn^3(K_3) & \longrightarrow & H_\fn^3(K_2) & \longrightarrow & H_\fn^3(K_1) & \longrightarrow & H_\fn^3(K_0) 
		\end{array}
	\end{displaymath}
	Notice that  for every $(\mu,\nu)\in \Omega,\, H_\fn^2 (R)_{(\mu,\nu)}=0,$ hence $H_\fn^2(K_j)_{(\mu,\nu)} =0$ for all $j=0,\ldots,3.$	Moreover $H_\fn^3(K_q)\simeq K_{3-q}^\ast [2-3a,2-3b]$ for every $q=0,\ldots,3.$ Therefore, this spectral sequence at step two in bidegree $(\mu,\nu)\in \Omega$ gives
	\begin{displaymath}
		\begin{array}{cccc}
			0 & 0& 0& 0\\
			0 & 0& 0&  0\\
			0 & 0& 0&  0\\
			H_0^\ast[2-3a,2-3b] & H_1^\ast[2-3a,2-3b]  &  H_2^\ast[2-3a,2-3b]  & H_3^\ast[2-3a,2-3b] 
		\end{array}
	\end{displaymath}
	By comparing the two spectral sequences, one has $H_2=H_3=0$ and for all $(\mu,\nu)\in \Omega$, we have
	$$  H_\fn^1(R/I)_{(\mu,\nu)} \simeq  (H_1)^\ast_{(3a-\mu-2,3b-\nu-2)}=\Homgr_R((H_1)_{(3a-\mu-2,3b-\nu-2)},k)$$
	as claimed.
\end{proof}

Next, let throughout $I^{\sat}:=I:\fn^{\infty}$ denote the saturation of $I$ with respect to $\fn=(x,y)\cap (u,v)$.

\begin{prop} \label{proposition4} With the above notation, one has
	\begin{enumerate}
		\item[{\rm (i)}] If $0<b<2a$ then $\dim_k \Syz(I)_{(1,0)}=\deg (\Bc) - \dim_k (R/I^{\sat})_{(2a-3,2b-2)}.$
		\item[{\rm (ii)}] If $0<a<2b$ then $\dim_k \Syz(I)_{(0,1)}=\deg (\Bc) - \dim_k (R/I^{\sat})_{(2a-2,2b-3)}.$
	\end{enumerate}
\end{prop}
\begin{proof}
	(i) Observe that since $(B_1)_{(a+1,b)}=0$, Lemma~\ref{lemma2} and \eqref{eq:les} imply that
	\begin{align*}
		\dim_k \Syz(I)_{(1,0)}&=\dim_k (Z_1)_{(a+1,b)}=\dim_k(H_1)_{(a+1,b)}\\
		&=\dim_k H_\fn^1(R/I)_{(2a-3,2b-2)}\\
		&=\deg (\Bc) - \dim_k (R/I^{\sat})_{(2a-3,2b-2)}.
	\end{align*}	
	(ii) is proved similarly.
\end{proof}

\subsection{Birationality of bidegree $(1,1)$ maps} As noticed at the beginning of Section \ref{sec:lowdegdim}, 
the birationnality of bidegree (1,1) maps can be easily characterized by means of linear syzygies. Below, we reprove this fact using Proposition \ref{proposition4}.

\begin{prop}\label{prop:1-1} Let $\phi:\PP^1\times \PP^1 \dto \PP^2$ be a dominant rational map given by bihomogeneous polynomials $f_0,f_1,f_2$ of bidegree $(1,1)$. The following are equivalent:
	\begin{itemize}
		\item[(i)] $\phi$ is  birational,
		\item[(ii)] the polynomials $f_0,f_1,f_2$ have a nonzero bidegree $(1,0)$ syzygy,
		\item[(iii)] the polynomials $f_0,f_1,f_2$ have a nonzero bidegree $(0,1)$ syzygy.
	\end{itemize}
\end{prop}
\begin{proof} The map $\phi$ is birational if and only if $\deg(\phi)=1$ and by the degree formula \eqref{degree-formula} this is equivalent to having $\deg (\Bc)=\sum_{x\in \Bc} e_x(I)=1$. Now, by Proposition \ref{proposition4} with $a=b=1,$ we have
	\begin{align*}
		\dim_k \Syz(I)_{(1,0)}& =\deg (\Bc) - \dim_k (R/I^{\sat})_{(-1,0)}=\deg (\Bc),\\
		\dim_k \Syz(I)_{(0,1)}&=\deg (\Bc) - \dim_k (R/I^{\sat})_{(0,-1)}=\deg (\Bc).
	\end{align*}
	Therefore, $\phi$ is birational if and only if $\dim_k \Syz(I)_{(1,0)}=1$, or equivalently if and only if $\dim_k \Syz(I)_{(0,1)}=1$.
\end{proof}
 The above birationality criterion can be translated into a numerical effective test. For that purpose, set
$$f_i(x,y;u,v):=c_{i,0}xu+c_{i,1}xv+c_{i,2}yu+c_{i,3}yv.$$
We seek a triple of polynomials $g_0,g_1,g_2$ that are linear forms in $x,y$ (or equivalently $u,v$) and such that $\sum g_if_i\equiv 0$. Such a triple can be found as elements in the kernel of a matrix $M$ whose columns are filled with the coefficients of the polynomials
$$ xf_0, yf_0, xf_1, yf_1, xf_2, yf_2$$
in a basis of bihomogeneous polynomials of bidegree $(2,1)$, typically
$$ x^2u, x^2v, y^2u , y^2v, xyu, xyv.$$
The matrix $M$ is hence the following $6\times 6$-matrix
\begin{equation}\label{eq:M11}
	M=\left[
	\begin{array}{cccccc}
		c_{0,0} & 0          & c_{1,0}   & 0                & c_{2,0} & 0       \\
		c_{0,1} & 0           & c_{1,1}   & 0               & c_{2,1} & 0 \\
		0      & c_{0,2} & 0            & c_{1,2}        & 0         & c_{2,2} \\
		0  & c_{0,3}  & 0            & c_{1,3}        & 0 & c_{2,3}      \\
		c_{0,2}  & c_{0,0} & c_{1,2}  & c_{1,0}         & c_{2,2} & c_{2,0} \\
		c_{0,3}  & c_{0,1} & c_{1,3}    & c_{1,1}       & c_{2,3}          & c_{2,1}
	\end{array}
	\right].	
\end{equation}
As a consequence, in Proposition \ref{prop:1-1}, we could add as a fourth item the statement that $\det(M)=0$.

\subsection{Birationality of bidegree $(1,2)$ maps} 

Before providing our birationality criteria in this case, we establish the following technical lemma.

\begin{lem} \label{lemma34} Let $\phi:\PP^1\times \PP^1 \dto \PP^2$ be a dominant rational map defined by bihomogeneous polynomials $f_0,f_1,f_2\in R$ of bidegree $(1,2)$ without common factor in $R\setminus k$. Set $I=(f_0,f_1,f_2)\subset R$.  
	Then, we have	
	\begin{itemize}
		\item[(i)] $\dim_k \Syz(I)_{(0,1)}=\deg (\Bc)-2,$
		\item[(ii)] $\dim_k \Syz(I)_{(1,1)}=\deg (\Bc).$
	\end{itemize}
\end{lem} 
\begin{proof} (i) 
	Since $(a,b)=(1,2)$, Proposition \ref{proposition4} shows that
	$$\dim_k \Syz(I)_{(0,1)}=\deg \Bc- \dim_k(R/I^{\sat})_{(0,1)}.$$
	If $\deg(\Bc)=1$, then the base scheme of $\phi$ consists of a single simple point. Therefore $I^{\sat}=(x,u)$ up to a coordinate change, hence $\dim_k(R/I^{\sat})_{(0,1)}=1$ and we deduce that  
	there is no nonzero syzygy of bidegree $(0,1)$, as claimed. 
	
	Now, we assume that $\deg(\Bc)\geq 2$. Since $\dim_k(R_{(0,1)})=2$, it suffices to show that $I^{\sat}_{(0,1)}=0$. 
	Thus, suppose that $I^{\sat}_{(0,1)}\neq 0$; without loss of generality we may assume that $u\in I^{\sat}$. Now, since $\deg (\Bc)\geq 2,$ there exists a form $q(x,y)$ of bidegree $(2,0)$ such that $I^{\sat} \subset (q,u).$ But 
 since $ f_i \in I^{\sat}$, we have 
	$$f_i=a_i q+b_iu, \ \  i=0,1,2.$$
	As $\deg(f_i)=(1,2)$, we deduce that $a_0=a_1=a_2=0$ and that $u$ divides $f_i$ for all $i=0,1,2$; this is a contradiction.
	
	\smallskip
	
	(ii) By inspecting the shifts of bidegrees in the Koszul complex of the sequence $f_0,f_1,f_2$, and taking into account that the $f_i$'s are of bidegree $(1,2)$, we observe that	
	$$\dim_k \Syz(I)_{(1,1)}=\dim_k(Z_1)_{(2,3)}=\dim_k(H_1)_{(2,3)}.$$
Applying Lemma \ref{lemma2} (we have $(-1,1)\in \Omega$), we get the equality 
$$\dim_k(H_1)_{(2,3)}=\dim_k H_{\fn}^1(R/I)_{(-1,1)}.$$
Now, the exact sequence \eqref{eq:les} restricted to bidegree $(-1,1)$ yields the equality
 $$\dim_k H_{\fn}^1(R/I)_{(-1,1)}=\deg (\Bc)- \dim_k(R/I^{\sat})_{(-1,1)}=\deg(\Bc)$$
and the claimed equality is proved. 
\end{proof}

\begin{thm} \label{Theorem8} Let $\phi:\PP^1\times \PP^1 \dto \PP^2$ be a dominant rational map given by bihomogeneous polynomials $f_0,f_1,f_2\in R$ of bidegree $(1,2)$ without common factor in $R\setminus k$. Setting $I=(f_0,f_1,f_2)\subset R$, the following are equivalent:
	\begin{itemize}
		\item[(i)] $\phi$ is birational,
		\item[(ii)]  $\deg(\Bc)=3,$ and hence $I$ is generically a complete intersection,
		\item[(iii)] $\dim_k \Syz(I)_{(0,1)}=1$,
		\item[(iv)] $\dim_k \Syz(I)_{(1,1)}=3$.
	\end{itemize}
\end{thm}
\begin{proof} Since (ii) is equivalent to both (iii) and  (iv) by Lemma~\ref{lemma34}, it suffices to show that (i) and (ii) are equivalent.

Now, by the degree formula \eqref{degree-formula}, we have
	\begin{equation}\label{eq:degformtmp}
	\sum_{x\in \Bc}e_x(I)=4-\deg(\phi)\leq 3.	
	\end{equation}
Moreover, by property of the Hilbert-Samuel multiplicity we also have (see, e.g., \cite[\S4.5]{BH93}) $\deg(\Bc)\leq \sum_{x\in \Bc}e_x(I)$	with equality if and only if $I$ is generically a complete intersection. Therefore, if $\deg(\Bc)=3$ then $\sum_{x\in \Bc}e_x(I)=3$, so that $I$ is generically a complete intersection, and from \eqref{eq:degformtmp} we deduce that $\deg(\phi)=1$, i.e.~$\phi$ is birational. Thus, we have just proved that (ii) implies (i). To prove the converse, suppose that $\deg(\Bc)\neq 3$. Then, necessarily, $\deg(\Bc)\leq 2$ and this implies that $I$ is generically a complete intersection. Therefore $\deg(\Bc)=\sum_{x\in \Bc}e_x(I)\leq 2$ and hence $\phi$ cannot be birational by \eqref{eq:degformtmp}. It follows that (i) is equivalent to (ii). 
\end{proof}

\begin{rem}\label{rem:bideg11}\rm  
Item (iii) provides us with a minimal syzygy of bidegree $(0,1)$ so that $u(\sum_{i=0}^2a_if_i)=v(\sum_{i=0}^2b_if_i)$ for some $a_i$'s and $b_i$'s in $k$. It follows that
there exist three polynomials $p,q,r$ of bidegree $(1,1)$ such that $I=(pu,pv,qu+rv)$.
Therefore, $I$ is a perfect ideal generated by the $2$-minors of the matrix
$$M:=\begin{pmatrix} v & q\\   -u&  r\\ 0  &  -p \end{pmatrix}.$$
Thus, one could add yet another equivalent condition to Theorem~\ref{Theorem8}, namely that the ideal $I$ has a free $R$-resolution of the form
	\begin{equation*}
		\xymatrix{ 0 \ar[r] & R(-1,-3)\oplus R(-2,-3)\ar[r] ^{\qquad \; M}& R(-1,-2)^3  \ar[r] & R\ar[r]& R/I \ar[r]& 0.}
	\end{equation*}
	Note that, in this format, three independent 
$(1,1)-$syzygies of $I$ are 
$$(xv,-xu,0), (yv,-yu,0),(q,r,-p),$$ 
the first two being non-minimal.	Hence, in contrast to the spirit of Proposition~\ref{prop:1-1}, in item (iv) of the above theorem $\Syz(I)_{(1,1)}$ 
is not spanned by $3$ minimal syzygies of bidegree $(1,1)$.
\end{rem}
\begin{cor}\label{cor:12} If $\phi$ is  birational, then $\dim_k \Syz(I)_{(1,0)} =0.$
\end{cor}
\begin{proof}
	This is an immediate consequence of Remark~\ref{rem:bideg11}, 
	using the fact that $p$ and  $qu+rv$ have no proper common factor as $I$ has codimension $2$.
\end{proof}

\begin{rem}\rm Theorem \ref{Theorem8} and Corollary \ref{cor:12} provide another illustration that the converse of Proposition \ref{Plinear rank} does not hold, here for some dominant rational maps from $\PP^1\times \PP^1$ to $\PP^2$. 
\end{rem}

\subsection{Birational maps of bidegree $(2,2)$} \label{sec:22case}

Unlike the cases of rational maps of bidegree (1,1) or (1,2), the linear syzygies associated to a given parameterization are not enough to give birational criterion in higher bidegrees. Yet, in the case of bidegree $(2,2)$, we are able to describe a complete listing of such birational maps. 

\medskip

Let $\phi:\PP^1\times \PP^1 \dto \PP^2$ be a dominant rational map given by bihomogeneous polynomials $f_0,f_1,f_2\in R$ of bidegree $(2,2)$.  We set  $I=(f_0,f_1,f_2)$ and we denote by $\Bc$ the base scheme of $\phi$ which is assumed to be zero-dimensional (i.e.~supported on a finite set of points). The degree formula yields the equality 
\begin{equation}\label{eq:degform22}
	\deg(\phi)=8-\sum_{x\in \Bc} e_x(I).
\end{equation}
And since $\deg(\phi)\geq 1$, we deduce that
\begin{equation*}
1\leq \deg \Bc\leq \sum_{x\in \Bc} e_x(I)\leq 7.	
\end{equation*}

For a codimension $2$ bihomogeneous prime ideal $\pp\supset I$, we will set $d_\pp:=\dim_k(R_{\pp}/I_{\pp})$  (``point degree'')  and let as before $e_\pp$ denote the Hilbert-Samuel multiplicity of $I$ on $\pp$. 
As is well-known, $e_\pp\geq d_\pp$, with equality if and only if $I_{\pp}$ is a complete intersection (a fact we have already used in the proof of Theorem \ref{Theorem8}).

By abuse, one may think of $\pp$ as belonging to $\Bc$;
as such it is the defining prime ideal of a point $p\in \PP^1\times \PP^1$.
Fix one such $\pp$. By changing coordinates, there is no loss of generality in assuming $p= (0,1)\times (0,1)$, i.e., $\pp=(x,u)$.

First, we remark the following :
\begin{lem} \label{lemmadeg6}
Assume the above notation. If $\phi$ is birational then $\deg \Bc \leq 6$. Moreover, if $\deg \Bc =6$ then $I$ is perfect with a minimal resolution of the form :
\begin{displaymath}
		\xymatrix{ 0 \ar[r] & R(-3,-3)^2\ar[r] & R(-2,-2)^3  \ar[r] & I\ar[r]& 0,}
\end{displaymath}
\end{lem}

\begin{proof} Two general $k$-linear combinations of the $f_i$'s define a scheme on the support of $\Bc$ plus an additional simple point $q$ that does not share any coordinate with the base points (the argument is similar to the one given in the last paragraph of the introduction of Section \ref{sec:lowdegdim}). If $\deg \Bc=7$ then choose a point $p$ in the support of $\Bc$, and if $\deg \Bc =6$ then take $p$ as the point for which the Hilbert-Samuel multiplicity is not equal to its degree. After a linear change of coordinates on the source and target spaces, we may and will assume that $q=(1,0)\times (1,0)$,  $p= (0,1)\times (0,1)$ and $(f_1,f_2)=(y,v)\cap J\cap K\cap L$ where  $J\subset (x,u)$ is unmixed with associated primes corresponding to the support of $\Bc$, while $K$ (respectively $L$) is $(x,y)$-primary (respectively $(u,v)$-primary) of degree 4 and generically a complete intersection (i.e.~the image of $K$ (respectively $L$) in $k(u,v)[x,y]$ (respectively $k(x,y)[u,v]$) is a complete intersection). 

Now, we observe that the defining ideal of $\Bc$ is either $J$ if $\deg \Bc =7$, or either $J:(x,u)$ if $\deg \Bc =6$. The latter is a consequence of liaison (see for instance \cite[\S 21.10]{Eis95}). 
 Furthermore, we have that $K:(x,y)=(x,y)^2$ and $L:(u,v)=(u,v)^2$. Therefore
 $$
 (f_1,f_2):(xu,yv)=(J:(x,u))\cap (x,y)^2\cap (u,v)^2
 $$ 
 and in particular $((f_1,f_2):(xu,yv))_{(2,2)}=(J:(x,u))_{(2,2)}$.
 By rewriting $f_1=Axu+Byv$ and $f_2=Cxu+Dyv$, we get 
 $$
 (f_1,f_2):(xu,yv)=(f_1,f_2,AD-BC).
 $$ 
 
 Now, $f_3\in I\subseteq (J:(x,u))\cap (x,y)^2\cap (u,v)^2$.  As the $f_i$'s are linearly independent, $f_3$ is a nonzero multiple of $AD-BC$ modulo $f_1$ and $f_2$, hence we should have $I=(f_1,f_2):(xu,yv)$ which is unmixed of degree 6. This rules out the possibility of having $\deg \Bc =7$ and concludes the proof.
\end{proof}
We now discuss how the strict inequality $e_\pp> d_\pp$ reflects in the form of the generators of $I$.
For this, we resort to explicit computations on the affine piece $y=v=1$.
Now, one has
$I\subset (x,u)_{(2,2)}$ and the latter is spanned by the monomials $$xyv^2, y^2uv, x^2v^2, xyuv, y^2u^2, x^2uv, xyu^2, x^2u^2.$$
Therefore, for $j=0,1,2$, $g_j:=f_j(x,1,u,1)$ is a $k$-linear combination of the monomials $\{x,u,x^2,xu,u^2,x^2u,xu^2, x^2u^2\}$.
Set $J:=(g_0,g_1,g_2)\subset S:=k[x,u]$.
 
Consider the total order $>$ on the monomials of $S$ by decreeing
	$$1>x>u>x^2>xu>u^2>x^3>x^2u>xu^2>u^3>x^4>x^3u>x^2u^2> \cdots.$$
	Let $\ind(J)$ be the initial ideal $J$ with respect to the order $>$. Therefore $d_\pp =\dim (k[x,u]/\ind(J)).$

\begin{lem} \label{lemma16}  With the above notation, the equality $d_\pp=e_\pp$ holds except in the following cases:
	\begin{enumerate}
		\item[{\rm (i)}] $ J=(x,u)^2,$ in which case $d_\pp=3,\ e_\pp=4,$
		\item[{\rm (ii)}] $ J=(x^2+\lambda u^2+\mu xu, x^2u,xu^2), \lambda \neq 0,$ in which case $d_\pp=5,\ e_\pp=6.$
		\item[{\rm (iii)}] $  J=(xu+ \mu u^2+xu^2,  x^2+ \alpha xu+ \beta u^2, x^2u^2),\mu\neq 0,$ in which case one has
		$4\leq \deg \Bc\leq 5$ and $\sum_{x\in \Bc} e_x(I)\leq 6.$
	\end{enumerate}
\end{lem}
\begin{proof} Write $d_\pp=n,\ n\geq 1.$
We will argue in terms of the initial ideal $\ind(J)$.
	
We first consider the easy case where $x\in \ind(J)$ or $u\in  \ind(J).$ The argument will be totally symmetric in the two cases, so it suffices to consider one of them, say, $x\in \ind(J).$ 
Then $u^n\in \ind(J)$ and $u^{n-1}\notin \ind(J),$ hence $ \ind(J)=(x,u^n).$ 
Letting then $h_1,h_2\in J$ be polynomials such that $\ind(h_1)=x, \ind(h_2)=u^n,$  $\{h_1,h_2\}$ will be a Gr\"obner basis of $J$. Therefore, $J$ is a complete intersection, hence $d_\pp=e_\pp.$

\smallskip

Next consider the case where neither $x\in \ind(J)$ nor $u\in\ind(J).$ Notice that $x^2\in \ind(J)$, as otherwise $\ind(J) \subset (u)$.
We now analyse all possibilities:  both $xu$ and $u^2$ belong to $\ind(J)$; $xu\in \ind(J)$ and $u^2\notin \ind(J)$; $xu\notin \ind(J)$ and $u^2\in \ind(J)$; and  neither $xu$ nor $u^2$ belongs to $\ind(J)$, respectively.

\smallskip 

\textbf{\underline{Case 1:}} $xu,u^2\in \ind(J).$  By the chosen order of the monomials, one must have $\{x^2,xu,u^2\}\subset J.$ But certainly $J\subset (x^2,xu,u^2)$ since it does not contain either $x$ or $u$ and further $\{x^2u,xu^2,x^2u^2\}\subset (x^2,xu,u^2)$.
This shows that that $J=(x,u)^2,$ in which case $d_\pp=3$ and $e_\pp=4.$

\smallskip
	
\textbf{\underline{Case 2:}} $xu\in \ind(J)$ and $u^2\notin \ind(J).$ Hence $e_\pp\geq d_\pp\geq 4.$ Write 
	\begin{equation*}
		g_0 = x^2+ \lambda u^2+xul, \ 
		g_1 =   xu+\mu u^2+xul', \
		g_2 =   xul''
	\end{equation*}
where $(\lambda,\mu)\neq (0,0)$ and $l,l',l''$belong to the $k$-vector space spanned by $x,u,xu$.

If $\mu=0$ (hence $\lambda\neq 0$) then $xu\in J$ and $x^2+\lambda u^2\in J$, thus showing that $e_\pp\leq 4,$ and hence that $d_\pp=e_\pp=4.$ 

If $\mu\neq 0,$ pick explicit coefficients for $l,l',l''$:
	\begin{align*}
		g_0 &= x^2+ \lambda u^2+a_1x^2u+b_1xu^2+c_1x^2u^2, \\
		g_1 &=   xu+ \mu u^2+a_2x^2u+b_2xu^2+c_2x^2u^2, \\
		g_2 &=   a_3x^2u+b_3xu^2+c_3x^2u^2.
	\end{align*}
We are led to consider the following sub-cases:
	\begin{enumerate}
		\item[(a)]  If $(a_3,b_3)=(0,1),$ then $xu^2\in J,$ hence $x^2u\in J.$ It follows  that $J\supset (x^2+\lambda u^2,xu+\mu u^2)$ which show that $e_\pp=4.$
		\item[(b)]  If $(a_3,b_3)=(1,\nu),$ then we write:
		\begin{align*}
			g_0& = x^2+ \lambda u^2+axu^2+bx^2u^2,\\
			g_1 &=   xu+ \mu u^2+cxu^2+dx^2u^2,\\
			g_2 &=   x^2u+\nu xu^2+ex^2u^2.
		\end{align*}
Therefore $xg_1-g_2=(\mu-\nu)xu^2+(c-e)x^2u^2+dx^3u^2.$ If $\mu\neq \nu$ then $xu^2\in J.$ Hence $J\supset (x^2+\lambda u^2,xu+\mu u^2)$ which shows that $e_\pp=4.$ Conversely, if $\mu=\nu=1,$ then we write:
		\begin{align*}
			g_0& = x^2 -\lambda xu+\alpha xu^2+\beta x^2u^2,\\
			g_1 &=   xu+  u^2+\gamma xu^2+\delta x^2u^2,\\
			g_2 &=   x^2u+ xu^2+\xi x^2u^2.
		\end{align*}
Taking $g_2-ug_0\in J$ gives  $xu^2\in J,$ therefore $J\supset (x^2+\lambda u^2,xu+ u^2)$ which shows that $e_\pp=4.$ 
		\item[(c)]  If $(a_3,b_3)=(0,0),$  then we write:
		\begin{align*}
			g_0& = x^2+ \lambda u^2+ax^2u+bxu^2,\\
			g_1 &=   xu+ \mu u^2+cx^2u+dxu^2,\\
			g_2 &=   x^2u^2.
		\end{align*} 
Since $x^2u^2\in I,$ $\Bc$ has only one prime $\pp,$ which shows that $\deg \Bc=d_\pp$ and $\sum_{x\in \Bc} e_x(I)=e_\pp.$	It is easy to see that $x^4\in J,$ hence $4 \leq d_\pp\leq 5.$ Moreover $(x,u)^4\subset J,$ hence $x^2u+\mu xu^2\in J.$ We can write
		\begin{align*}
			g_0& = x^2+ \lambda u^2+(b-a\mu)xu^2,\\
			g_1 &=   xu+ \mu u^2+(d-c\mu)xu^2,\\
			g_2 &=   x^2u^2.
		\end{align*} 
If $b-a\mu=d-c\mu=0$ then $J\supset (x^2+\lambda u^2, xu+\mu u^2)$ which shows that $d_\pp=e_\pp=4.$ Conversely, if $(b-a\mu)^2+(d-c\mu)^2\neq 0$ then 
	$$J\supset (xu+ \mu u^2+xu^2,  x^2+ \alpha xu+ \beta u^2),$$
	therefore $e_\pp\leq 6.$
	\end{enumerate}
	
	\textbf{\underline{Case 3:}} $xu\notin \ind(J)$ and $u^2\in \ind(J).$ Therefore $d_\pp\geq 4.$  We write 
	\begin{align*}
		g_0 &= x^2+ \lambda xu + xu(ax+bu+cxu),\\
		g_1 &=   u^2+xu(a'x+b'u+c'xu), \\
		g_2 &=   xu(a''x+b''u+c''xu).
	\end{align*}
Since $g_1-a'ug_0=u^2(1+\alpha x+\beta x^2+\gamma xu+\delta x^2u)\in J,$ hence $u^2\in J,$ therefore $x^2+\lambda xu+a x^2u\in J.$ It follows that
	$$(u^2,x^2+\lambda xu +a x^2u)\subset J.$$
Consider the codimension $2$ homogeneous ideal $G=(u^2, x^2z+\lambda xuz+a x^2u)\subset T:=k[x,u,z]$ obtained by homogenizing the two generators of the leftmost ideal in the above inclusion.
Then $T/G$ is a complete intersection of degree $6$ supported on two points in $\PP^2$, namely, $=(0:0:1)$ and $=(1,;0:0).$ Letting $\qq_1, \qq_2$ denote the respective defining prime ideals, one has $d_{\qq_2}=e_{\qq_2}=2,$ hence $d_{\qq_1}=e_{\qq_1}=4.$
	Since $e_\pp\leq  e_{\qq_1}=4,$ therefore $ e_\pp=d_\pp =4.$\\
	
	\textbf{\underline{Case 4:}} $xu,u^2\notin \ind(J).$ It is seen that $d_\pp\geq 5.$  We can write
	\begin{align*}
		g_0 &= x^2+ \lambda u^2 + \mu xu+ xu(ax+bu+cxu),\\
		g_1 &=   a_1x^2u+b_1xu^2+c_1x^2u^2, \\
		g_2 &=   a_2x^2u+b_2xu^2+c_2x^2u^2,
	\end{align*}
	where $\lambda \neq 0.$  
	Again, consider the following sub-cases:
	\begin{enumerate}
		\item[(a)] If $a_1=a_2=0,$ then $(b_1,b_2)\neq (0,0)$. Therefore, we obtain
		\begin{align*}
			g_0& = x^2+ \lambda u^2 + \mu xu+ ax^2u,\\
			g_1 &=   xu^2,\\
			g_2 &=   x^2u^2.
		\end{align*}
		Since  $g_2=xg_1,$ hence $J=(g_0,g_1)$ is a complete intersection, therefore $e_\pp=d_\pp$.
		\item[(b)] If $a_1=1$ and $b_2=0$. We obtain
		\begin{align*}
			g_0& = x^2+ \lambda u^2 + \mu xu+\alpha x^2u+\beta xu^2,\\
			g_1 &=   x^2u+\gamma xu^2,\\
			g_2 &=   x^2u^2.
		\end{align*}
		If $\gamma=0,$ then $J=(g_0,g_1)$ is a complete intersection, hence $e_\pp=d_\pp$. We deduce that $\gamma\neq 0$ and hence $xu^3=\gamma^{-1}(ug_1-g_2)\in J.$ 
		Write
		\begin{align*}
			g_0& = x^2+ \lambda u^2 + \mu xu+\alpha x^2u,\\
			g_1 &=   x^2u+\gamma xu^2,\\
			g_2 &=   x^2u^2.
		\end{align*}
		It is easy to see that $u^4\in J.$  If $u^3\notin \ind(J),$ then $d_\pp =e_\pp=7.$ Conversely, if $u^3\in \ind(J),$  then $5\leq d_\pp\leq 6.$ Moreover, since $xu^2\notin \ind(J)$, we obtain $\ind(J)=(x^2,u^3).$ Therefore, there exists a Gr\"obner basis of $J$ of two polynomials, hence $J$ is a complete intersection.
		\item[(c)] If $a_1=1$ and $b_2=1$. We obtain
		\begin{align*}
			g_0& = x^2+ \lambda u^2 + \mu xu+ \alpha x^2u^2,\\
			g_1 &=   x^2u+\beta x^2u^2=x^2u(1+\beta u),\\
			g_2 &=   xu^2+\gamma x^2u^2=xu^2(1+\gamma x).
		\end{align*}
		It follows that $x^2u,xu^2\in J.$ We write 
		$$ug_0=x^2u+\lambda u^3+\mu xu^2+a x^2u^3\in J,$$
		which shows that $u^3\in J.$ It follows that $d_\pp=5$ and $e_\pp=6.$ In this case
		$$J=(x^2+\lambda u^2+\mu xu, x^2u,xu^2), \lambda \neq 0.$$
	\end{enumerate}
\end{proof}

Now, we derive consequences of the above technical lemma and the degree formula \eqref{eq:degform22}.

\begin{cor}\label{proposition5} Let $\phi:\PP^1\times \PP^1 \dto \PP^2$ be a dominant rational map given by bihomogeneous polynomials $f_0,f_1,f_2\in R$ of bidegree $(2,2)$ without common factor in $R\setminus k$. If $\phi$ is birational then $\deg \Bc= 6$.
\end{cor}

\begin{thm} \label{Theorem20}
	Let $\phi:\PP^1\times \PP^1 \dto \PP^2$ be a dominant rational map given by bihomogeneous polynomials $f_0,f_1,f_2\in R$ of bidegree $(2,2)$ without common factor in $R\setminus k$. Assume that the point in $\Bc$ with the largest multiplicity is the point $\pp:=(0,1)\times (0,1)$. Then, with the notation established in the beginning of the section, $\phi$ is birational if and only if $\deg \Bc=6$ and  $J=(x,u)^2$ or $$J=(x^2+\lambda u^2+\mu xu, x^2u,xu^2), \ \lambda \neq 0.$$
\end{thm}

\begin{proof}  First, assume that $\phi$ is birational. Corollary \ref{proposition5} shows that  $\deg(\Bc)=6$. Moreover, by \eqref{eq:degform22} we have $\sum_{x\in \Bc} e_\pp=7$ and hence Lemma~\ref{lemma16} implies that $(d_\pp,e_\pp)=(3,4)$ or $(d_\pp,e_\pp)=(5,6)$, that is to say $J=(x,u)^2$ or $J=(x^2+\lambda u^2+\mu xu, x^2u,xu^2), \lambda \neq 0$, as claimed.

For the converse it suffices to prove that if $\deg \Bc=6$ and  $J=(x,u)^2$ or $J=(x^2+\lambda u^2+\mu xu, x^2u,xu^2), \lambda \neq 0,$ then $\phi$ is birational, i.e.~$\sum_{x\in \Bc}e_x(I)=7$. We now analyse these two possibilities.
	
	\smallskip
	
\textbf{\underline{Case 1:}}  Suppose that $\deg \Bc=6$ and $J=(x,u)^2\in \Bc.$ Let $\{\pp_1,\ldots,\pp_r\}$ denote the primes of $\Bc$ other than $\pp=(x,u)$.
Since $d_{\pp}=3$, then  $\sum_{i=1}^{r} d_{\pp_i}=3.$ 
Thus, in order to have the total sum $\sum_{x\in \Bc}e_x(I)=7$ is now tantamount to having $J_{\pp_i}$ a complete intersection for every $i=1,\ldots,r.$ But this is clear because Lemma~\ref{lemma16} shows that otherwise $d_{\pp_i}=3$ and $e_{\pp_i}=4$, for every $i=1,\ldots,r.$
	
	\smallskip
	
\textbf{\underline{Case 2:}}  Suppose that $\deg \Bc=6$ and $J=(x^2+\lambda u^2+\mu xu, x^2u,xu^2), \lambda \neq 0.$ Since $\deg \Bc=6,$  Lemma~\ref{lemma16} implies that $\Bc$ has only one prime $\qq$  other than $\pp=(x,u)$, with $d_\qq=1$ and $e_\qq=1.$ Therefore $\sum_{x\in \Bc}e_x(I)=7$ as required.
\end{proof}

By  Lemma~\ref{lemmadeg6}, if $\phi:\PP^1\times \PP^1 \dto \PP^2$ is a birational map defined by bihomogeneous polynomials $f_0,f_1,f_2\in R$ of bidegree $(2,2)$ without common factor in $R\setminus k$, then, $I=(f_0,f_1,f_2)$ is perfect ideal with exactly two minimal syzygies, of bidegree $(1,1)$. Indeed, 
the free resolution of $I$ is of the form 
\begin{displaymath}
\xymatrix{ 0 \ar[r] & R(-3,-3)^2\ar[r] ^M& R(-2,-2)^3  \ar[r] & I\ar[r]& 0.   }
\end{displaymath}

To understand the shape of the matrix $M$, we consider three cases.
	
	\smallskip
	
\textbf{\underline{Case 1:}} Suppose that $\Bc=\{\pp_1, \pp_2,\pp_3\} \subset \PP^1\times \PP^1$ with $d_{\pp_1}=3$, $d_{\pp_2}=2$ and $d_{\pp_3}=1$.  By a suitable coordinate change,  one can assume without loss of generality that the three primes are $\pp_1= (x,u)$, $\pp_2=(y,v)$ and $\pp_3=(x+y, u+ v)$.
Accordingly
$$I\subset (x,u)^2\cap (y^2,v)\cap (x+y, u+ v).$$
Now,  $I$ is generated by elements of bidegree $(2,2)$.
A computation with \texttt{Macaulay2} gives that
$$\{x^2v(u+v),\, xuv(x+y),\, y^2u^2- xyuv\}$$
are the only forms of bidegree $(2,2)$ in the variables $x,y,u,v$.
In particular, $I$ must be contained in the ideal generated by these three forms, and hence coincides with it.

These three forms are the $2$-minors of the following $3\times 2$ matrix 
	$$M=\begin{pmatrix} u(x+y) & yu\\   -x( u+ v) &  - yu\\ 0  &  xv  
		\end{pmatrix}.$$

	\smallskip
			
\textbf{\underline{Case 2:}}  Suppose that $\Bc=\{\pp_1,\pp_2,\pp_3,\pp_4\}\subset \PP^1 \times \PP^1$ with $d_{\pp_1}=3$ and $d_{\pp_i}=1$, $i=2,3,4$. By the same token as in the first case, we may assume that $\pp_1= (x,u),\ \pp_2=(y,v)$, $\pp_3=(x+y,u+v)$ and $\pp_4=(ax+y, \alpha u+ v)$, for suitable coefficients $a,\alpha\in k$ with $(a,\alpha )\neq (1,1)$.
Accordingly,
		$$I\subset (x,u)^2\cap (y,v)\cap (x+y,u+v)\cap (ax+y, \alpha u+v).$$
Repeating the same computational device as in the first case, one obtains that $I$ is generated by the following forms
	{\small 
		$$\{(yu-xv)(\alpha xu+xv),  
		(axu+yu)(yu-xv),(\alpha - a)x^2uv+(\alpha-1) xyuv-(a-1)x^2v^2\}.
		$$	
		}
	Once again, one can verify that these forms are the $2$-minors of the $3\times 2$ matrix 
		$$M=\begin{pmatrix} (ax+y)u & (a-1)xv\\  -x(\alpha u+ v)   &   -(\alpha -1 )xv\\ 0&   yu-xv  
		\end{pmatrix}.$$
		
If $\alpha \neq a$, the three minors have no factor in common and it follows that $M$ provides the free resolution of $I$. If $\alpha =a$, then all elements of bidegree $(2,2)$ in 
$(x,u)^2\cap (y,v)\cap (x+y,u+v)\cap (ax+y, a u+v)$ are multiple of $yu-xv$, contradicting the hypothesis that the $f_i$'s have no common factor.
		\smallskip
			
\textbf{\underline{Case 3:}} Suppose that $\Bc=\{\pp,\qq\}\subset \PP^1 \times \PP^1$ with $d_{\pp}=5$ and $d_{\qq}=1$. Always by the same token,  we may assume that $\pp=(x,u)$ and $\qq=(y,v)$. Accordingly, one has $I\subset (x^2+\lambda u^2+\mu xu, x^2u,xu^2)\cap (y,v)$, for suitable $\lambda\in k \setminus 0$. Since then 
		$$I\subset (x^2v^2+\lambda y^2u^2+\mu xyuv, x^2uv,xyu^2),$$
		it must be generated by these three forms of bidegree $(2,2)$.
	
As before, direct inspection shows that $I$ is perfect with syzygy matrix
		$$M=\begin{pmatrix} xu& 0\\ - xv  &  yu \\ \lambda yu+\mu xv & xv \end{pmatrix}.$$

\section{Modeling: tensor-product maps in the plane}\label{sec:applications}

In this section we will explore the consequences of our previous results to the field of geometric modeling. Indeed, in this field bigraded rational maps are intensively used to describe parameterizations of curves, surfaces and volumes, including plane parameterizations. For that purpose, the Bernstein basis is preferred to the usual power basis for representing polynomials. Recall that the homogeneous Bernstein polynomials are defined by the formula 
$$B_i^n(x,y)=\binom{n}{i}y^i(x-y)^{n-i}.$$
They are homogeneous of degree $n$ and any homogeneous polynomial of degree $n$ can be written as a linear combination of them. Consequently, a bihomogeneous polynomial of bidegree $(a,b)$ can be written as a linear combination of all the products $B_i^a(x,y)B_j^b(u,v)$, $i=0,\ldots,a$ and $j=0,\ldots,b$. Rational maps written in this basis are dubbed \emph{tensor-product B\'ezier parameterizations}.

It turns out that an important property of tensor-product B\'ezier parameterization is to guarantee their birationality. Moreover, an even more important property is to preserve this birationality property during a design process, that is to say when the coefficients of the defining polynomials are continuously modified (see e.g.~Figure \ref{fig:ex12}). In what follows, we will show how Theorem \ref{Theorem8} and Proposition \ref{prop:1-1} allow to translate the detection of birationality as rank decision problems in the case of tensor-product parameterizations of bidegree (1,1) and (1,2).

\subsection{Plane tensor-product parameterizations}

For defining a bigraded rational map of bidegree $(a,b)$ in Bernstein form we need to introduce a collection of \emph{control points} $P_{i,j}=(x_{i,j}, y_{i,j}) \in \RR^2$ and their associated weights $w_{i,j} \in \RR$. The map is then defined as
\begin{eqnarray}\label{eq:bigradedmap}
	\phi : \PP^1\times \PP^1 & \dashrightarrow & \PP^2 \\ \nonumber
	(x:y)\times (u:v) & \mapsto & \left( 
	\sum_{i,j} w_{i,j} B^{a,b}_{i,j}:
	\sum_{i,j} w_{i,j}x_{i,j}B^{a,b}_{i,j} : 
	\sum_{i,j} w_{i,j}y_{i,j}B^{a,b}_{i,j}
	       \right)
\end{eqnarray}
where $B^{a,b}_{i,j}:=B_i^a(x,y)B_j^b(u,v)$. Observe that $\phi$ ``interpolates'' the control points, in the sense that
$$\phi\left((1:i)\times(1:j)\right)=(w_{i,j}:w_{i,j}x_{i,j}:w_{i,j}y_{i,j}), \ i=0,1, \ j=0,1.$$
In addition, if all the weights are equal to 1 then $\sum_{i,j} w_{i,j} B^{a,b}_{i,j}=1$, so that the control points $P_{i,j}$ fully control the map $\phi$. 

In general, the control points are the only coefficients of the map $\phi$ that are modified during a hand-design process because they really provide an intuitive way to reshape the parameterization $\phi$. The weights are hidden behind and not used as an intuitive design tool. When the control points of a given birational parameterization are moved, then the new parameterization is in general no longer a birational parameterization. Below, we will illustrate how the weights of the map $\phi$ can be changed in order to retrieve a birational map without touching again to the control points modified by the designer.

\subsection{Bilinear tensor-product parameterizations}
Consider a rational map as defined in \eqref{eq:bigradedmap} with $(a,b)=(1,1)$. By Proposition \ref{prop:1-1}, this rational map will be birational if and only if there exists a syzygy of bidegree (1,0), or equivalently a syzygy of bidegree (0,1). Writing this condition under a linear system in the Bernstein basis, we obtain the following matrix whose kernel yields those bidegree (1,0) syzygies:
$$M:=\left(\begin{array}{cccccc} 
	 {x_{0,0}}  {w_{0,0}}&
     0&
      {y_{0,0}}  {w_{0,0}}&
     0&
     w_{0,0}&
     0\\
      {x_{0,1}}  {w_{0,1}}&
     0&
      {y_{0,1}}  {w_{0,1}}&
     0&
     w_{0,1}&
     0\\
     \frac{1}{2}  {x_{1,0}}  {w_{1,0}}&
     \frac{1}{2}  {x_{0,0}}  {w_{0,0}}&
     \frac{1}{2}  {y_{1,0}}  {w_{1,0}}&
     \frac{1}{2}  {y_{0,0}}  {w_{0,0}}&
     \frac{1}{2}  {w_{1,0}}&
     \frac{1}{2}  {w_{0,0}}\\
     \frac{1}{2}  {x_{1,1}}  {w_{1,1}}&
     \frac{1}{2}  {x_{0,1}}  {w_{0,1}}&
     \frac{1}{2}  {y_{1,1}}  {w_{1,1}}&
     \frac{1}{2}  {y_{0,1}}  {w_{0,1}}&
     \frac{1}{2}  {w_{1,1}}&
     \frac{1}{2}  {w_{0,1}}\\
     0&
      {x_{1,0}}  {w_{1,0}}&
     0&
      {y_{1,0}}  {w_{1,0}}&
     0&
     w_{1,0}\\
     0&
      {x_{1,1}}  {w_{1,1}}&
     0&
      {y_{1,1}}  {w_{1,1}}&
     0&
     w_{1,1}\\
\end{array}\right).$$
As a consequence, the map is birational if and only if $\det(M)=0$. Now, using the Laplace expansion formula of determinants by $3\times 3$-blocks with respect to columns 1,3,5 and 2,4,6, we get the condition :
{\small
\begin{multline*}
\left( w_{1,0}w_{0,1} \left| \tilde{P}_{0,0}\tilde{P}_{0,1}\tilde{P}_{1,0}\right|\cdot \left|\tilde{P}_{0,1}\tilde{P}_{1,0}\tilde{P}_{1,1}\right|
- w_{1,1}w_{0,0} \left| \tilde{P}_{0,0}\tilde{P}_{0,1}\tilde{P}_{1,1}\right|\cdot \left|\tilde{P}_{0,0}\tilde{P}_{1,0}\tilde{P}_{1,1}\right|
\right)\times \\ 
w_{0,0}w_{0,1}w_{1,0}w_{1,1} =0	
\end{multline*}
}

\vspace{-1em}

\noindent where $\tilde{P}_{i,j}$ is the vector $(x_{i,j},y_{i,j},1)=(P_{i,j},1)$. Weights are in general assumed to be nonzero; therefore, under this assumption we recover the following condition that already appeared in the recent paper \cite{SZ}:
\begin{equation}\label{eq:criterion11}
\frac{w_{1,0}w_{0,1}}{w_{1,1}w_{0,0}}=
\frac{\left| \tilde{P}_{0,0}\tilde{P}_{0,1}\tilde{P}_{1,1}\right|\cdot \left|\tilde{P}_{0,0}\tilde{P}_{1,0}\tilde{P}_{1,1}\right|}{\left| \tilde{P}_{0,0}\tilde{P}_{0,1}\tilde{P}_{1,0}\right|\cdot \left|\tilde{P}_{0,1}\tilde{P}_{1,0}\tilde{P}_{1,1}\right|}.	
\end{equation}
From here, it appears clearly that given the control points, a suitable modification of a single weight so that \eqref{eq:criterion11} holds, allow to obtain a birational map \cite{SZ}.

\subsection{Bidegree (1,2) tensor-product parameterizations}
Now, consider a bilinear rational map as defined in \eqref{eq:bigradedmap} with $(a,b)=(1,2)$. By our previous results, this rational map will be birational if and only if there exists a syzygy of bidegree (0,1). Proceeding as in the previous case of bilinear maps, we obtain the following multiplication matrix
\begin{equation}\label{eq:genMatrix}
M :=
\left( 
\begin{array}{cccccc}
w_{0,0}x_{0,0} & 0 & w_{0,0}y_{0,0} & 0 & w_{0,0} & 0 \\	
\frac{2}{3}w_{0,1}x_{0,1} & \frac{1}{3}w_{0,0}x_{0,0} & \frac{2}{3}w_{0,1}y_{0,1} & \frac{1}{3}w_{0,0}y_{0,0} &\frac{2}{3}w_{0,1} & \frac{1}{3}w_{0,0} \\	
\frac{1}{3}w_{0,2}x_{0,2} & \frac{2}{3}w_{0,1}x_{0,1} & \frac{1}{3}w_{0,2}y_{0,2} & \frac{2}{3}w_{0,1}y_{0,1} & \frac{1}{3}w_{0,2} & \frac{2}{3}w_{0,1} \\ 	
0 & w_{0,2}x_{0,2} & 0 & w_{0,2}y_{0,2} & 0 & w_{0,2} \\
w_{1,0}x_{1,0} & 0 & w_{1,0}y_{1,0} & 0 & w_{1,0} & 0 \\	
\frac{2}{3}w_{1,1}x_{1,1} & \frac{1}{3}w_{1,0}x_{1,0} & \frac{2}{3}w_{1,1}y_{1,1} & \frac{1}{3}w_{1,0}y_{1,0} &\frac{2}{3}w_{1,1} & \frac{1}{3}w_{1,0} \\	
\frac{1}{3}w_{1,2}x_{1,2} & \frac{2}{3}w_{1,1}x_{1,1} & \frac{1}{3}w_{1,2}y_{1,2} & \frac{2}{3}w_{1,1}y_{1,1} & \frac{1}{3}w_{1,2} & \frac{2}{3}w_{1,1} \\ 	
0 & w_{1,2}x_{1,2} & 0 & w_{1,2}y_{1,2} & 0 & w_{1,2} 
\end{array}
\right).	
\end{equation}
It is $8\times 6$-matrix and $\rank(M)<6$ if and only if the corresponding map is birational. Similarly, the analysis of bidegree $(1,1)$ syzygies leads to a square $12\times 12$-matrix with the property that its rank drops by 3 if and only if the corresponding map is birational. Therefore, the decision of birationality is not given by a single polynomial condition as in the previous case of bilinear maps. Nevertheless, our syzygy-based formulation of birationality by means of the rank of the matrix $M$ translates birationality decision to a \emph{rank decision problem}. This opens a bridge to the field of numerical linear algebra where a huge amounts of works on this problem have been done during the last decades. In the following, we illustrate this link on an example with the help of a recent algorithm for structured low-rank approximation \cite{SS15}.

\medskip

We start with the canonical non-rational tensor-product parameterization of the plane, i.e.~all the weights are set to 1 and the control points have a rectangular shape. More precisely, we set $P_{i,j}=(x_{i,j},y_{i,j})=(i,j)$ for all $i=0,1$ and $j=0,1,2$, which is illustrated on the left side of Figure \ref{fig:ex12}.  
This initial parameterization is birational, which can be checked by observing that the matrix $M$ specialized to this setting
has rank 5. 
Now, as illustrated in Figure \ref{fig:ex12}, suppose that these control points are "moved" in order to reach the following new coordinates :
$$P_{0,0}=(0,0), P_{0,1}=(-1/2,1), P_{0,2}=(0,2),$$
$$P_{1,0}=(2,-1/2), P_{1,1}=(5/2,1),P_{1,2}=(2,5/2).$$
\begin{figure}
\centering
\includegraphics[width=11cm]{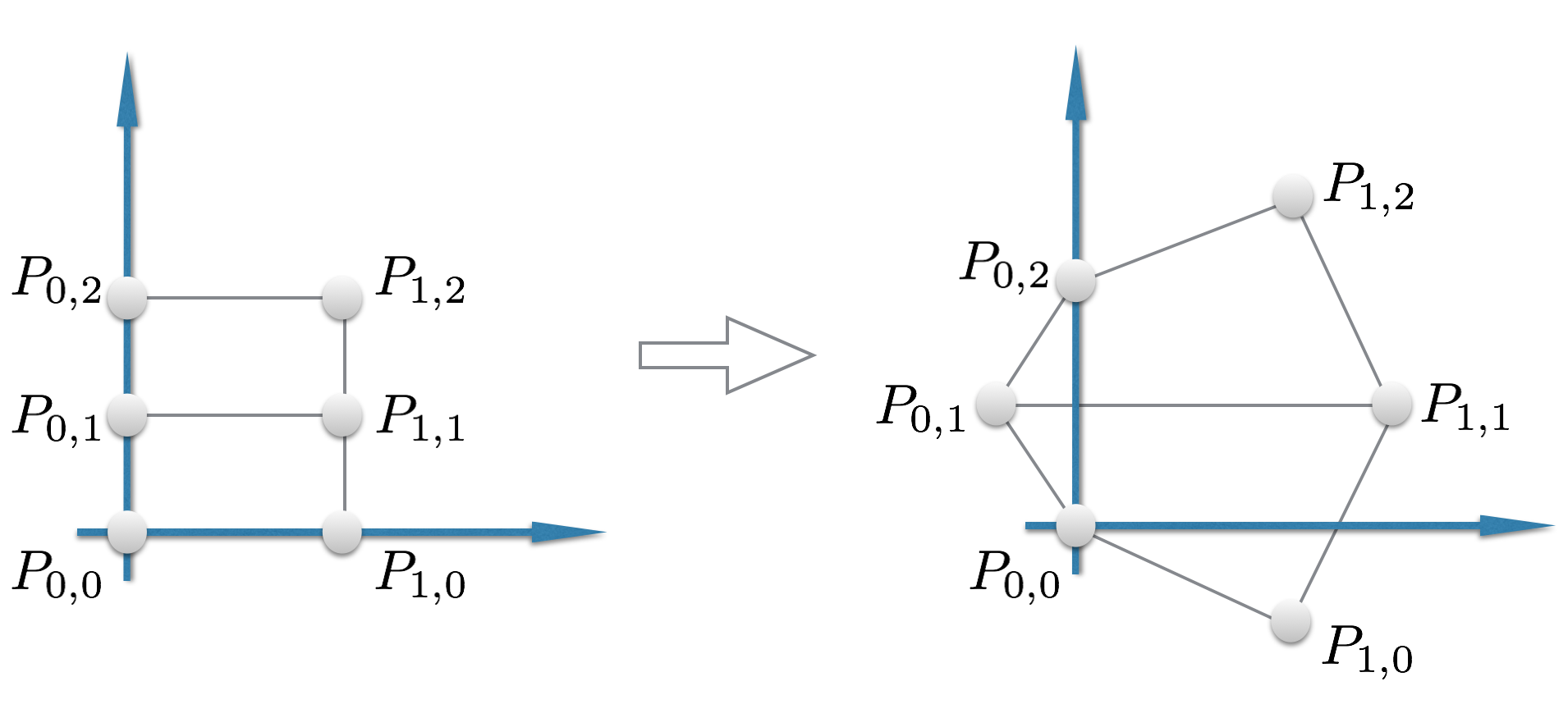}
\caption{Change of control points in a tensor-product parameterization of bidegree (1,2).}\label{fig:ex12}
\end{figure}

\vspace{-1em}

\noindent If the weights are left unchanged, i.e.~all equal to 1, then this new parameterization is no longer birational. Indeed, it is straightforward to check that the matrix $M$ specialized with these new control points and all weights equal to 1 has rank 6. 
So, we aim at changing the weights $w_{i,j}$, without changing the control points, so that the parameterization becomes rational. For that purpose, we will apply the structured low-rank approximation algorithm developed in  \cite{SS15}.

\medskip

Given a matrix $M$, the basic idea of structured low-rank approximation is to compute a matrix $M'$ of given rank $r$ in a linear subspace $E$ of matrices such that the distance, in the sense of the Frobenius norm, between $M$ and $M'$ is small. Such an algorithm, based on Newton-like iterations, is given in \cite{SS15}. In our context, by \eqref{eq:genMatrix} the matrix $M$ can be written as
$$M=w_{0,0}E_{0,0}+w_{0,1}E_{0,1}+w_{0,2}E_{0,2}+w_{1,0}E_{1,0}+w_{1,1}E_{1,1}+w_{1,2}E_{1,2}$$
where the $E_{i,j}$'s are matrices of size $8\times 6$ whose entries only depend on the control points. These latter define a linear subspace of matrices and we are looking for a matrix $M'$ such that $M'$ belongs to this linear subspace and its rank is lower or equal to 5. Thus, applying the algorithm in \cite{SS15}, we find the following weights, up to numerical precision :
\medskip
\begin{center}
\begin{tabular}{ccc}
	                   $w_{0,0}\approx0.949726775368655$, & & 
	                   $w_{0,1}\approx1.0867765091791244$, \\
	                   $w_{0,2}\approx0.9521336386754828$, & & 
	                   $w_{1,0}\approx1.0233828904581144$,\\
	                   $w_{1,1}\approx0.9458573850234199$, & & 
	                   $w_{1,2}\approx1.0259764181881534$,
\end{tabular}	
\end{center}    
\medskip
Therefore, by modifying the weights with the above values, the parameterization becomes birational ``up to numerical precision'', which means in practice that its 5-minors yield inversion formulas for almost all points, up to numerical precision.

\section{Acknowledgments}

 All authors are partially supported by the Math-AmSud program called SYRAM (Geometry of SYzygies of RAtional Maps with applications to geometric modeling). The fourth named author was additionally supported by a CNPq grant (300586/2012-4) and by a ``p\'os-doutorado no exterior''. The fifth named author was also additionally supported by a CNPq grant  (302298/2014-2) and a PVNS Fellowship from CAPES (5742201241/2016). Most of the computations were done with  \texttt{Macaulay2} \cite{M2}, as an important aspect of this work.


\end{document}